\documentclass[12pt]{article}
\usepackage[nottoc]{tocbibind}
\usepackage[T1]{fontenc}
\usepackage[utf8]{inputenc}
\usepackage{enumerate}
\usepackage{changes}
\usepackage{hyperref}
\usepackage{amsmath}
\usepackage{amsfonts}
\usepackage{amssymb}
\usepackage{amsthm}
\usepackage{newlfont}
\usepackage{mathtools}
\usepackage[
left=1.1in, right=1.1in]{geometry}

\usepackage{authblk}
\usepackage{xcolor}
\usepackage{thmtools}
\usepackage{tikz}

\theoremstyle{plain}
\newtheorem{theorem}{Theorem}

\newtheorem{corollary}[theorem]{Corollary}
\newtheorem{lemma}[theorem]{Lemma}

\newtheorem{assumption}[theorem]{Assumption}

\theoremstyle{definition}

\newtheorem{example}{Example}
\newtheorem*{remark}{Remark}
\newtheorem*{acknowledgements}{Acknowledgements}

\newcommand{\RR}{\mathbb{R}}
\newcommand{\NN}{\mathbb{N}}
\newcommand{\ZZ}{\mathbb{Z}}

\newcommand{\cC}{\mathcal{C}}
\newcommand{\drm}{\mathrm{d}}
\newcommand{\euler}{\mathrm{e}}
\newcommand{\cD}{\mathcal{D}}
\newcommand{\cE}{\mathcal{E}}

\DeclareMathOperator{\Deg}{Deg}

\DeclareMathOperator{\arsinh}{arsinh}

\DeclareMathOperator{\diam}{diam}

\newcommand{\eat}[1]{}

\let\oldint\int
\renewcommand{\int}{\oldint\limits}

\newcommand{\Hmm}[1]{\leavevmode{\marginpar{\tiny%
			$\hbox to 0mm{\hspace*{-0.5mm}$\leftarrow$\hss}%
			\vcenter{\vrule depth 0.1mm height 0.1mm width \the\marginparwidth}%
			\hbox to
			0mm{\hss$\rightarrow$\hspace*{-0.5mm}}$\\\relax\raggedright #1}}}

\begin{document}

\title{Off-diagonal upper heat kernel bounds on graphs with unbounded geometry
}
\author{Christian Rose\thanks{christian.rose@uni-potsdam.de}
}
\affil[]{Institut f\"ur Mathematik, Universit\"at Potsdam, 		14476  Potsdam, Germany}
\date{\today}
\maketitle
\begin{abstract}
Results regarding off-diagonal Gaussian upper heat kernel bounds on discrete weighted graphs with possibly unbounded geometry are summarized and related. 

After reviewing uniform upper heat kernel bounds obtained by Carlen, Kusuoka, and Stroock, 
the universal Gaussian term on graphs found by Davies is addressed and related to corresponding results in terms of intrinsic metrics. Then we present a version of Grigor'yan's two-point method with Gaussian term involving an intrinsic metric. A discussion of upper heat kernel bounds for graph Laplacians with possibly unbounded but integrable weights on bounded combinatorial graphs preceeds the presentation of compatible bounds for anti-trees, an example of combinatorial graph with unbounded Laplacian.
Characterizations of localized heat kernel bounds in terms of intrinsic metrics and universal Gaussian are reconsidered. Finally, the problem of optimality of the Gaussian term is discussed by relating Davies' optimal metric with the supremum over all intrinsic metrics.
\\

	\noindent \textbf{Keywords:} graph, heat kernel, Gaussian bound, unbounded geometry
	\\
	\noindent \textbf{2020 MSC:} 39A12, 35K08, 60J74
\end{abstract}

\tableofcontents

\section{Introduction}
On spaces with appropriate notion of Laplace operator acting on functions, one defines the heat kernel to be the minimal fundamental solution of the corresponding heat equation. 
Heat kernels are very sensitive with respect to the underlying geometry of the space, and encode local and global analytic, geometric and stochastic properties. 

An explicit representation formula for the heat kernel is known only in rare situations, among them the heat kernel $p$ on $\RR^n$ given by the Gau\ss{}-Weierstra\ss{} function
\[
p_t(x,y)=\frac{1}{\sqrt{4\pi t^n}}\exp\left(-\frac{|x-y|^2}{4t}\right)
\]
for all $x,y\in\RR^n$ and $t>0$.

The lack of explicit formulas leads to a  fundamental interest in good quantitative bounds of the heat kernel on more general spaces in terms of their geometric properties. Building upon the ingenious work of Aronson, de Giorgi, Nash, Moser, and Li and Yau \cite{Nash-58, Moser-64, Aronson-67, LiYau-86}, this led to a fruitful and intense research area in different contexts over the last decades. Of particular interest are geometric characterizations of upper or lower bounds on heat kernels comparable to the Gau\ss{}-Weierstra\ss{} function, called Gau\ss{}ian bounds. We refer to the exponential term as Gaussian term. Major results include corresponding bounds on Riemannian manifolds, fractals, fractional Laplace operators, and general Dirichlet forms, see \cite{Grigoryan-94, Grigoryan-09, CoulhonG-98, Varopoulos-85,CarlenKS-87,Coulhon-92,SaloffCoste-92,Grigoryan-94,
SaloffCoste-92a,Delmotte-99,SaloffC-01,BCS, BarlowChen-16, GrigoryanHH} for a non-exhaustive list of references.
\\

This exposition concentrates on reviewing Gaussian upper heat kernel
 upper bounds for the continous-time heat kernel of discrete weighted graphs with infinite vertex set and possibly unbounded geometry. Loosely speaking, a graph is a discrete set of vertices which are connected by non-directed weighted edges without loops and a measure on the vertices, cf.~Section~\ref{sec:graphs} for precise definitions. A Laplace-type operator acting on functions defined on the vertices is given by mimicking the mean value property for harmonic functions in the combinatorial neighborhood of a given vertex. 
 
 For a prototypical example of a space with explicit heat kernel, consider the graph with vertex set being the set of integers $\ZZ$, where we connect $x,y\in\ZZ$ by an edge if and only if their distance is one, i.e., if and only if $|x-y|=1$ and zero otherwise. Define the operator
\[
\Delta u(n)=2u(n)-u(n-1)-u(n+1),
\]
for $n\in\ZZ$ and $u\in \cC(\ZZ)$. Its Friedrichs extension is a non-negative, self-adjoint and bounded operator in $\ell^2(\ZZ,2)$.
By using Fourier series and identities for Bessel functions, the heat kernel $p$ associated to $\Delta$ has the explicit representation 
\[
p_t(x,y)=\frac{1}{\sqrt\pi \Gamma(d+1/2)}\euler^{-2t}t^d\int_{-1}^1(1-z^2)^{d-1/2}\euler^{2tz}\drm z,
\]
for $x,y\in\ZZ$ and $t>0$, \cite[p.~376]{AbramowitzStegun-65}.
The number $d=d(x,y)$ denotes the combinatorial distance between $x$ and $y$, that is, the number of edges of the shortest path connecting $x$ and $y$. Although explicit, it is desirable to have upper bounds on $p$ which are comparable to the Gauss-Weierstrass function. An essentially exact estimate was obtained in \cite[Theorem~3.5]{Pang-93}: there exists a constant $c>0$ such that 
\[
\frac{c^{-1}}{\sqrt t\vee d}\exp\left(-2t\ \zeta\left(\frac{d}{2t}\right)\right)\leq p_t(x,y)\leq \frac{c}{\sqrt t\vee d}\exp\left(-2t\ \zeta\left(\frac{d}{2t}\right)\right)
\]
for $x,y\in\ZZ$, $d,t>0$,
where 
\[
\zeta\big(x\big)
=
x\arsinh\left(x\right) 
+1-\sqrt{x^2+1},\qquad x\geq 0.\]
Moreover, for $x,y\in\ZZ$, $x\neq y$ we have 
\[
t\ p_t(x,y)\to 0,\qquad t\to 0,
\]
and
\[
2t\ \zeta\left(\frac{d}{2t}\right)\sim \frac{d^2}{4t},\qquad t\to \infty,
\]
where $\sim$ means that the left-hand side divided by the right-hand side converges to one.
Clearly, the long-time asymptotics of the heat kernel on the integers and on Euclidean space are comparable. In contrast, the short-time behavior differs from the Varadhan asymptotics for the Gauss-Weierstrass function, and this fact holds for general discrete graphs \cite{KLMST-16}.


Analogous to the case of manifolds, one aims at geometric characterizations of heat kernel bounds which are comparable to the behavior of the heat kernel on the integers. 
Unlike Riemannian manifolds, graphs come without a canonical choice for the vertex measure, and the Laplacian generates a non-local Dirichlet form on the combinatorial structure. A choice for a  metric is also necessary in order to formulate Gaussian bounds and to encode geometric properties. 

If the graph is very tame, i.e., has very bounded geometry,  the characterization of upper heat kernel bounds is, apart from certain optimality questions, nowadays quite well-understood in terms of the combinatorial distance, see, e.g., the monographs \cite{Woess-00,Barlow-book, Grigoryan-graphbook} and the references therein. This differs tremendously if the geometry of the graph is not bounded, which can happen in different ways: the combinatorial graph, the edge weights and measure need not be bounded above or below, but the Laplacian can be bounded even if the combinatorial geometry is not uniformly bounded. 
\eat{Certain problems can be solved by the use of intrinsic metrics, which deliver the existence of good cut-off functions. }
 \\

The present paper aims at summarizing results regarding upper heat kernel bounds in the context of graphs with possibly unbounded geometry. It addresses the following:
\begin{itemize}
\item Uniform, i.e., functional analytic characterizations of off-diagonal upper heat kernel bounds.
\item Universal heat kernel bounds on graphs.
\item Sharp bounds on heat kernels defined on operators with mean-integrable weights on combinatorial graphs.
\item Heat kernel bounds on anti-trees.
\item Recent characterizations of localized heat kernel bounds valid for graphs with respect to intrinsic metrics.
\item  The problem of optimality of involved metrics.
\eat{ developped in \cite{KellerRose-23, KellerRose-22a, Rose-24}. 
}
\end{itemize}

In order to stay concise, the focus lies on results concerning upper heat kernel bounds instead of the presentation of techniques. This will allow for a precise comparison of the existing literature on the topic. Few proofs will be presented in order to effectively compare the results and to clarify particular statements. For a deeper study, the reader's attention will be drawn to the cited literature. If it turns out that some contribution is missing, we would be greatful for references and would include them in the future.

The Gaussian upper bounds treated here are universal in the sense that they hold for all discrete graphs, and are compared to the Gaussian estimate on the integers. However, there are graphs which have a much better Gaussian estimate. Graphs which obey a fractal-like global structure like the discrete Sierpinski gasket provide examples of graphs with sub-Gaussian behavior of the heat kernel, i.e., faster than the Euclidean Gaussian. An interested reader is referred to \cite{BarlowB-92,GrigoryanT-01,Barlow-book, GrigoryanHH, MuruganSC-23,Murugan-24} and the references therein for a glympse on the related existing literature.

We do not address characterizations of full Gaussian bounds in terms of Harnack and Poincar\'e inequalities and volume doubling.  The characterization of full Gaussian heat kernel bounds in case of very bounded geometry was addressed in the fundamental work by Delmotte \cite{Delmotte-99}, cf.~\cite{Barlow-book}. It has recently been generalized to bounded Laplacian but possibly unbounded combinatorial geometry \cite{BarlowChen-16}. A corresponding characterization for graphs with unbounded geometry will appear in a forthcoming work.

\section{Laplace operator and  heat equation on graphs}\label{sec:graphs}

The terminology follows mostly \cite{KellerLW-21}.

Let $X$ be countable and
extend a function
 $m\colon X\to(0,\infty)$ of full support to a measure on $X$ via countable additivity.
A symmetric function $b\colon X\times X\to [0,\infty)$  with 
\[
b(x,x)=0\quad\text{and}\quad \deg_x:=\sum_{y\in X}b(x,y)<\infty
\]
for all $x\in X$ 
is called \emph{graph} over the measure space $(X,m)$. Often, such a graph is simply denoted by the triple $(X,b,m)$. We write $ x\sim y $ whenever $ b(x,y)>0 $ for $ x,y\in X $. A graph is called \emph{locally finite} if the set $\{y\in X\colon b(x,y)>0\}$ is finite for any $x\in X$. It is called connected if for any $x,y\in X$ there exists a finite sequence $x_1,\ldots,x_n\in X$, $n\in\NN$, such that $x_i\sim x_{i+1}$ for all $i\in\{1,\ldots, n-1\}$ and $x_1=x$ and $x_n=y$. 
\begin{center}
\emph{We assume in the following that all graphs are locally finite and connected.}
\end{center}
The Laplace operator $\Delta\colon \cC(X)\to\cC(X)$ on a graph acts on the set of functions $\cC(X)$ as
\[
\Delta f(x)=\frac1{m(x)}\sum_{y\in X} b(x,y)(f(x)-f(y)), 
\] 
for $ f\in \cC(X) $ and $  x\in X  $.
By local finiteness, $ \Delta $ maps the set of compactly supported functions $ \cC_{c}(X) $ into itself. Hence, the restriction of $ \Delta $ to $ \cC_c(X) $ is symmetric in $ \ell^{2}(X,m)$, the Hilbert space of square integrable functions on $(X,m)$ with respect to the inner product $\langle \cdot,\cdot\rangle_{\ell^2(X,m)}$.  By slight abuse of notation we also denote the closure of $\Delta\restriction_{\cC_c(X)}$ in $\ell^2(X,m)$ by   $\Delta\geq 0$.

The Dirichlet form corresponding to $\Delta$ is given by 
\[
\cE(f)=\frac{1}{2}\sum_{x,y\in X}b(x,y)(f(x)-f(y))^2
\]
for all $f$ in the domain $\cD(\cE)\subset \ell^2(X,m)$,  and form and operator are related via Green's formula
\[
\cE(f)=\sum_{x\in X}m(x)f(x)\Delta f(x)\], for all $v\in \cC(X)$. The energy measure $\Gamma$ of $\cE$ is given by
\[
\Gamma(f,g)(x)=\frac{1}{2m(x)}\sum_{x,y\in X}b(x,y)(f(x)-f(y))(g(x)-g(y)),
\] 
and we simply write $\Gamma(f)(x)=\Gamma(f,f)(x)$ for all $x\in X$ and $f,g\in \cC(X)$.

The (continuous-time) heat semigroup $(\euler^{-t\Delta})_{t\geq 0}$, acts in $ \ell^{2}(X,m) $ and has an integral kernel $p\colon [0,\infty)\times X\times X\to[0,\infty)$, called the \emph{heat kernel},
satisfying
\[
\euler^{-t\Delta}f(x)=\sum_{y\in X} m(x) p_t(x,y)f(y)
\]
for all $ x\in X, t\geq 0, f\in \ell^2(X,m) $.
For an initial condition $ f\in \ell^{2}(X,m) $, the function $ u=\euler^{-t\Delta}f $ solves the heat equation
\[
\frac{d}{dt}u=-\Delta u, \quad u(0,\cdot)=f.
\] 
A direct application of the representation of the heat semigroup shows 
\[
p_t(x,y)=\frac{1}{m(x)m(y)}\left\langle \delta_{x},\euler^{-t\Delta}\delta_{y}\right\rangle_{\ell^2(X,m)},
\]
where $\delta_{x}$ denotes the function which is one at $x\in X$ and zero otherwise. For basic properties of the heat kernel, see \cite{KellerLVW-15,KLMST-16,KellerLW-21}.

Heat kernels of graphs correspond to the transition density of the jump process on the underlying graph associated to the Laplacian $\Delta$. Two 
 examples for the measure $m$ are frequently discussed in the literature. The normalizing measure $m=\deg$ generates so-called constant speed random walks. Variable speed random walks on graphs correspond to the counting measure $m=1$.

Boundedness of the Laplacian of a graph $b$ over $(X,m)$ is determined by the boundedness of the weighted vertex degree
\[
\Deg(x)=\frac{\deg(x)}{m(x)}=\frac{1}{m(x)}\sum_{y\in X}b(x,y),\qquad x\in X.
\]
In fact, the Laplacian is a bounded operator if and only if $\sup_X\Deg<\infty$, see, e.g., \cite[Lemma~1]{Davies-93}. This means in particular that the Laplacian with respect to the normalizing measure is always bounded, while it might not with respect to the counting or a general measure. Classical results on heat kernel bounds only consider the normalizing measure. In the following, we will not impose that the Laplacian is bounded.


\section{Characterization of uniform off-diagonal bounds}\label{sec:uniform}

This section briefly addresses characterizations of uniform off-diagonal heat kernel bounds. Uniformity in the present context refers to uniform-in-time on-diagonal heat kernel bounds, i.e., the non-Gaussian term is spatially independent of the heat kernel.

In his famous work, Varopoulos \cite{Varopoulos-85} characterized polynomial uniform upper heat kernel bounds in terms of Sobolev inequalities for general Dirichlet forms. Sobolev inequalities imply Nash inequalities, such that the upper heat kernel bound follows from Nash's original ideas \cite{Nash-58}\footnote{Nash includes Stein's "quick proof" of the Nash inequality in Euclidean space.}. For the reverse implication, Varopoulos transfers Hardy-Littlewood theory to abstract semigroups. 

About the same time, Davies observed that off-diagonal upper heat kernel bounds of heat kernels of elliptic operators are a consequence of on-diagonal upper bounds \cite{Davies-87}. The latter are equivalent to log-Sobolev inequalities. Davies derives log-Sobolev inequalities for elliptic operators sandwiched by  control functions and derived their on-diagonal heat kernel bounds. Optimization with respect to the sandwiching function then yields off-diagonal bounds for the original heat kernel. This method is nowadays referred to as Davies' method\footnote{According to Davies, Varopoulos discovered the same method independently \cite{Varopoulos-89}.}.

 Carlen, Kusuoka, and Stroock observed that Davies' method can be combined with Nash's approach to heat kernel bounds to obtain a characterization of uniform off-diagonal heat kernel bounds in terms of Nash inequalities \cite{CarlenKS-87}. Their main result, originally formulated for general Dirichlet forms, reads in the present context of discrete graphs as follows.

\begin{theorem}[{\cite[Theorem~3.25]{CarlenKS-87}}]\label{thm:ckkw}
Let $n>0$. The following properties are equivalent:
\begin{enumerate}[(i)]
\item There is a constant $c>0$ such that the Nash inequality 
\[
 \|f\|_2^{2+4/n}\leq c\ \cE(f)\|f\|_1^{4/n}
\]
holds for all $f\in \cD(\cE)$.
\item
For any $\epsilon\in(0,1)$ there is a constant $C_\epsilon>0$ such that the heat kernel $p$ satisfies
\[
p_t(x,y)\leq \frac{C_\epsilon}{ t^{\frac{n}2}}\exp\left( -\frac{d_\cE'(x,y)^2}{4(1+\epsilon) t}\right),
\]
for all $x,y\in X$ and $t>0$,
{, where 
\[
d_\cE'(x,y):=\sup\{\psi(x)-\psi(y)\colon \psi\in \cD(\cE), \Lambda(\psi)\leq 1 \},
\]
where
\[
\Lambda(\psi)^2
=
\left\| 
\euler^{-2\psi}\Gamma(\euler^{\psi})
\right\|_\infty \ \vee \
\left\| 
\euler^{2\psi}\Gamma(\euler^{-\psi})
\right\|_\infty.
\]
}
\end{enumerate}
\end{theorem}

In the same paper, the authors obtain characterizations of heat kernel upper bounds with decay rates different for small and large times in terms of Nash inequalities involving different dimensions. 
More general versions of the Dirichlet form statement can be found in \cite{Coulhon-96,CKKW-21}. There, the functions $x\mapsto x^{1+2/n}$, and $x\mapsto x^{n/2}$ appearing in the Nash inequality and heat kernel bound, respectively, are replaced by more general functions which suitably translate into each other. 

Since the Nash inequalities are equivalent to other functional inequalities like Faber-Krahn or Sobolev inequalities, uniform heat kernel upper bounds can be characterized by these inequalities as well, cf.~\cite{BakryCLS-95}. 

A legitimate concern are sufficient geometric criteria for Nash or equivalent inequalities in order to obtain heat kernel estimates. It is known that isoperimetric-type inequalities imply any of the latter functional inequalities. Such isoperimetric inequalities for graphs with possibly unbounded geometry have been introduced in \cite{BauerKW-15}, where a Cheeger inequality has been derived in our general context in terms of intrinsic metrics. Localized Euclidean isoperimetric inequalities and their relation to Sobolev inequalities for general graphs can be found in \cite{KellerRose-22b}.

\section{Universal Gaussian terms}\label{sec:universal}
This section discusses the universal Gaussian term for the heat kernel $p$ of a given graph $b$ over the discrete measure space $(X,m)$. A priori, there is no canonical choice for a metric on a graph, and heat kernel bounds depend on the choice of a distance.  

In \cite{Davies-93}, Davies generalized Pang's heat kernel upper bound \cite{Pang-93} for the integers discussed in the introduction  to heat kernel bounds for graphs $(X,b,\deg)$ by Davies' method equipped with the combinatorial distance $d$. 
\begin{theorem}[{\cite[Theorem~10, Corollary~11]{Davies-93}}]If $b$ is a graph over $(X,\deg)$, the heat kernel satisfies
\[
p_t(x,y)\leq \frac{1}{\sqrt{\deg(x)\deg(y)}}\exp\left(-t\ \zeta\left(\frac{d(x,y)}{t}\right)\right),
\]
for all $x,y\in X$ and $t>0$. In particular, for any $\epsilon>0$ there is a $\delta>0$ such that if $d(x,y)/t<\delta$, we have
\[
p_t(x,y)\leq  \frac{1}{\sqrt{\deg(x)\deg(y)}}\exp\left(-\frac{d(x,y)^2}{(2+\epsilon)t}\right).
\]
\end{theorem}
Clearly, this theorem generalizes Theorem~\ref{thm:ckkw} in the sense that no assumptions on the regularity properties of the quadratic form are necessary. Additionally, as observed by Davies and Pang \cite{Davies-93,Pang-93}, the function $\zeta$ is essentially exact. The two instead of four in the exponential term appears because Davies' normalization differs from Pang's by a factor two.
\\

Davies' approach easily yields comparable heat kernel bounds for graphs with bounded vertex degree. In order to generalize this result to graphs with possibly unbounded vertex degree, we use the notion of intrinsic metrics. They have been used to prove diverse results, \cite{BauerKH-13, HaeselerKW-13, Folz-14b, Folz-14, BauerKW-15,Keller-15, HuangKS-20, KellerLW-21}. Originally  introduced in Sturm's fundamental work on strongly local Dirichlet forms in \cite{Sturm-94} they appear for general regular Dirichlet forms in \cite{FrankLenzWingert-14}.  
First ideas towards of intrinsic metrics on graphs can already be found in \cite{Davies-93a}, but they appear first explicitly  in \cite{Folz-11,GrigoryanHuangMasamune-12}.

A pseudo-metric $\rho\colon X\times X\to[0,\infty)$ for a graph $b$ over $(X,m)$ is called \emph{intrinsic} if 
\[
\sum_{y\in X} b(x,y)\rho^2(x,y)\leq m(x),
\]
for all  $ x\in X $. Note that this immediately yields  $\Gamma(\rho(x,\cdot))(y)\leq 1/2$ for all $x,y\in X$ by the triangle inequality. The {jump size} of an intrinsic metric $\rho$ is given by 
\begin{align*}
	S=\sup\{\rho(x,y)\colon b(x,y)>0, x,y\in X \}.
\end{align*}

\begin{example}\label{ex:pathdegreemetric}If the weighted vertex degree is bounded, then a suitable rescaling of the combinatorial distance is intrinsic. Let $b$ a graph $b$ over $(X,m)$ with a possibly unbounded Laplacian, and $S>0$. The function $\rho$ defined via
\[
\rho(x,y)=\inf_{x= x_0\sim\ldots\sim x_n=y}\sum_{k=1}^nS\wedge \sqrt{\frac{m(x_{k-1})}{\deg(x_{k-1})}}\wedge\sqrt{\frac{m(x_{k})}{\deg(x_{k})}}
\]
is always an intrinsic metric with jump size bounded by $S$ for $b$ over $(X,m)$.
\end{example}

For graphs supporting an intrinsic metric, we have the following generalization of Davies' and Pang's results. 
\begin{theorem}[{\cite[Cor.~1.1]{BauerHuaYau-17}}]\label{thm:BHY}
Let $b$ be a graph over $(X,m)$ and $\rho$ an intrinsic metric with jump size bounded by $S>0$. Then, we have
\[
p_t(x,y)\leq \frac{1}{\sqrt{m(x)m(y)}}\exp\left(-\frac{t}{S^{2}}\zeta\left(\frac{\rho(x,y)S}{t}\right)\right)
\]
for all $x,y\in X$ and $t>0$. In particular, for any $\epsilon>0$ there exists $\delta=\delta(S,\epsilon)>0$ such that if $\rho(x,y)/t<\delta$, then 
\[
p_t(x,y)\leq \frac{1}{\sqrt{m(x)m(y)}}\exp\left(-\frac{\rho(x,y)^2}{(2+\epsilon)t}\right).
\]
\end{theorem}

The proof of Theorem~\ref{thm:BHY} is very close to the proof of Theorem~\ref{thm:davies1}, where the  intrinsic property of the metric replaces the regularity assumption on the graph.

\section{From on-diagonal to off-diagonal bounds}

As discussed in Section~\ref{sec:uniform}, Davies' original approach to off-diagonal upper bounds for heat kernels includes the derivation of on-diagonal heat kernel bounds for sandwiched operators and an optimization procedure. A characterization of on-diagonal upper heat kernel bounds is used explicitly. 

A more direct approach, based on work on manifolds by Grigor'yan, is nowadays called Grigor'yan's two-point method \cite{Grigoryan-94,Grigoryan-99}\footnote{Grigor'yan refers to Ushakov \cite{Ushakov-80} for the origins of the method.}. Off-diagonal upper bounds in two points are a direct consequence of on-diagonal upper bounds in only these two points, if the latter are given in terms of sufficiently regular functions. This follows since the heat kernel can be estimated by products of the Gaussian term and weighted $L^2$-norms of the heat kernel involving a rooted Gaussian. Such weighted $L^2$-norms are then controlled iteratively by integral maximum principles and the assumed on-diagonal bound.
\\

Since our focus lies on heat kernel bounds for unbounded graph Laplacians, we quote here the most recent corresponding result, which is phrased in terms of regular functions. 
In the present context, a function $f\colon (0,\infty)\to(0,\infty)$ is called $(A,\gamma)$-regular for $A,\gamma\geq 1$,  if 
\[
\frac{f(\gamma s)}{f(s)}\leq A\frac{f(\gamma t)}{f(t)}
\]
for all $0<s< t$. An adaption of Grigor'yan's two-point method to graphs is as follows.

\begin{theorem}[{\cite[Theorem~1.1]{Chen-17}, cf.~\cite[Theorem~1.1]{Folz-11}}]\label{thm:ChenFolz}
Let $b$ be a graph over $(X,m)$ and $\rho$ an intrinsic metric with jump size one. Assume there are $x,y\in X$ and $f_z\colon (0,\infty)\to(0,\infty)$ being $(A,\gamma)$-regular for $A,\gamma\geq 1$, $z\in\{x,y\}$, such that there exists a constant $\delta\geq 0$ with
\[
\sup_{z\in \{x,y\},\ t>0} \frac{f_z(t)}{\euler^{\delta t}}<\infty.
\]
Assume 
\[
p_t(z,z)\leq \frac{1}{m(z)f_z(t)}
\]
for all $t>0$, $z\in\{x,y\}$. 
Then there are $c_1,c_2,c_3>0$ such that 
\[
p_t(x,y)\leq \frac{c_1}{\sqrt{m(x)m(y)}\sqrt{f_x(c_2t)f_y(c_2t)}}\exp\left(-c_3\frac{\rho(x,y)^2}{t}\right)
\]
for all $t\geq \rho(x,y)$.
\end{theorem}

The advantage of theorems in the above spirit is that off-diagonal bounds follow from on-diagonal bounds which depend only on the spatial location of the heat kernel. In contrast to the earlier version  \cite{Folz-11}, there is no restriction on the vertex measure.

Note that the nontrivial part of Theorem~\ref{thm:ChenFolz} lies in the range of $t$: assuming $t\geq \rho(x,y)^2$ instead would not require any assumption on $f$, and the result would just follow from the Chapman-Kolmogorov identity for the heat kernel.

In the next section we will apply Theorem~\ref{thm:ChenFolz} with the path degree metric to obtain optimal off-diagonal bounds for elliptic operators on integer lattices.

\section{Off-diagonal bounds on combinatorial graphs}
\label{sec:combinatorial}
As explained in the introduction, a graph can be unbounded in several ways: either the Laplacian is bounded and the combinatorial geometry is unbounded, or the Laplacian is unbounded the combinatorial geometry is bounded, or both are unbounded. In this section, we will focus on the case where the combinatorial geometry is tame, but the Laplacian is allowed to be unbounded. 

Let $b$ be a graph over $(X,m)$ and assume that there are $n\geq 2$ and $C>0$ such that the underlying combinatorial graph satisfies the following regularity conditions:
\begin{itemize}
\item $\mathrm{D}:=\sup_{x\in X}\#\{y\in X\colon y\sim x\}<\infty$,
\item for every $x\in X$ there is $r_x\geq 0$ such that $C^{-1}r^n\leq \sharp B(x,r)\leq Cr^n$ for all $r\geq r_x$,
\item  for every $x\in X$ there is $r_x\geq 0$ such that $$\left(\sum_{y\in B(x,r)}\|u\|^{\frac{n}{n-1}}\right)^{\frac{n-1}{n}}\leq C\sum_{(y,z)\in B_x(r)_E}|u(y)-u(z)|$$ for all $r\geq r_x$, where we denote by $B_E:=(B\times B)\cap E$ the set of edges lying inside the set $B\subset X$, where we recall the set of edges $E=\{(x,y)\in X\times X\colon b(x,y)>0\}$.
\end{itemize}
Above assumptions are satisfied, e.g., if the combinatorial graph is $\ZZ^n$, or if the combinatorial vertex degree is bounded and the isoperimetric constant is bounded. For elliptic operators on combinatorial graphs we use 
the so-called chemical distance $d_{\mathrm{ch}}$ defined by 
\[
d_{\mathrm{ch}}(x,y)=\inf_{x=x_0\sim\dots\sim x_n=y}\sum_{i=1}^n\sqrt{\frac{m(x_{i-1})\wedge m(x_i)}{b(x_{i-1},x_i)}}
\]
for any $x,y\in X$. The chemical distance is almost intrinsic for graphs with bounded vertex degree, what can be seen by comparing it with the path degree metric. 

Finally, denote for $p\in[1,\infty)$, $f\in\cC(X)$, $B\subset X$ or $B\subset X\times X$, and $0\leq \phi\in\cC(X)$,
\[
\|f\|_{p,B,\phi}=\left(\frac{1}{\# B}\sum_{x\in B}|f(x)|^p\phi(x)\right)^{\frac{1}{p}},\qquad \|f\|_{\infty,B}=\max_B|f|,
\]
and $ \|f\|_{p,B}=\|f\|_{p,B,1}$ for all $p\in[1,\infty]$. The following result can be seen as the graph version of a classical result of Trudinger~\cite{Trudinger-71}.
\begin{theorem}[{cf.~\cite[Theorem~3.2]{ADS-19}}]\label{thm:ADS}
Let $b$ a graph over $(X,m)$ satisfying the regularity properties above, $n\geq 2$, and $p,q,r\in(1,\infty]$ such that
\[
\frac{1}{r}+\frac{1}{p}\frac{r-1}{r}+\frac{1}{q}<\frac{2}{n}.
\]
Further, assume 
\begin{multline*}
\sup_{x\in X}\ \limsup_{s\to\infty}\  \|1\vee \deg{m}^{-1}\|_{p,B(x,s),m}\|1\vee \sum_{y\in X}b(\cdot, y)^{-1}\|_{q,B(x,s)}\\
\cdot\|1\vee m\|_{r,B(x,s)}\|1\vee {m}^{-1}\|_{q,B(x,s)}<\infty.
\end{multline*}
Then, there exist constants $C,\gamma>0$ sucht that for any $x\in X$ there exists $r_x\geq 0$ such that for any $t\geq r_x^2$ we have 
\[
p_t(x,y)\leq C\ \frac{\left(1+\frac{d_{\mathrm{ch}}(x,y)}{t}\right)^\gamma}{t^{\frac{n}{2}}}\exp\left(-2\mathrm{D} t \ \zeta\left(\frac{d_{\mathrm{ch}}(x,y)}{2\mathrm{D} t} \right)\right).
\]
\end{theorem}

\begin{proof}
Follow the complete proof of \cite[Theorem~3.2, p.~17]{ADS-19} without estimating the function $F$ in the notation of the latter article. Then, proceed as in the proof of \cite[Corollary~3.4]{ADS-19} without estimating the polynomial correction term.
\end{proof}

Theorem~\ref{thm:ADS} is valid for all graph Laplacians provided the underlying combinatorial lattice is regular enough. However, in the special case where the combinatorial graph is just the integer lattice,  the integrability assumption can be improved. Moreover, in this case we obtain a Gaussian in the spirit of Theorem~\ref{thm:ChenFolz}.
\begin{theorem}[{\cite[Prop.~2]{Bella-22}}]\label{thm:Bella}Let $n\geq 2$, $p\in(1,\infty)$, $q\in(n/2,\infty)$, $$\frac1p+\frac1q<\frac2{n-1},$$ and $b$ a graph over $(\ZZ^n,1)$ satisfying
\[
\sup_{x\in X}\limsup_{r\to\infty} \|b\|_{p,B_x(r)_E}<\infty,\qquad \sup_{x\in X}\limsup_{r\to\infty} \|b^{-1}\|_{q,B_x(r)_E}<\infty.
\]
Then, there exists a constant $c>0$ such that 
\[
p_t(x,y)\ \leq\  \frac{c}{t^{\frac{n}{2}}}\ \exp\left(-c\frac{\rho(x,y)^2}{t}\right),
\]
for all $x,y\in\ZZ^n$ and $t\geq \rho(x,y)$, where $\rho$ denotes the intrinsic path metric with jump size one, cf.~Example~\ref{ex:pathdegreemetric}.
\end{theorem}
\begin{proof}
First, by {\cite[Prop.~2]{Bella-22}}, we obtain
\[
p_t(x,y)\leq c \ t^{-\frac{n}{2}},
\]
for all $x,y\in\ZZ^n$ and $t\geq 1$. Since we trivially have $p_t(x,y)\leq 1$ for all $t\geq 0$, the above estimate on the heat kernel holds for all $t>0$ and $x,y\in\ZZ^n$  if we replace  $c$ by $1\vee c$. Hence, since monomials on the positive axis are regular, Theorem~\ref{thm:ChenFolz} involving $\rho_1$ yields the claim.
\end{proof}

As explained in \cite{Bella-22}, in contrast to Theorem~\ref{thm:ADS}, the integrability assumptions in the above theorem appear to be sharp in a certain sense. On the other hand, the Gaussian term is again not the universal Gaussian. It would be interesting if the mentioned theorems can be combined in order to get sharp integrability assumptions on regular combinatorial lattices and a heat kernel bound with the universal Gaussian term.

\section{Off-diagonal bounds for anti-trees}

Section~\ref{sec:combinatorial} describes examples of graphs having bounded combinatorial geometry but possibly unbounded Laplace operator. In this section, we discuss upper  heat kernel bounds for anti-trees, which constitute an important class of examples where the underlying graph has both unbounded combinatorial geometry and unbounded Laplacian.
Their prominent role is based on the delivery of examples showing similarities and disparities between graphs and manifolds.
Specifically, equipped with the combinatorial graph distance, they include stochastically incomplete examples of polynomial growth \cite{Wojciechowski11}  and positive spectral gap \cite{KellerLW-13}. This disparity was later resolved by the use of intrinsic metrics. 

	\begin{figure}[ht]
  \begin{minipage}[b]{\linewidth}
  \centering
\begin{tikzpicture}[scale=.6, yscale=.45]
  dot/.style={fill=blue,circle,minimum size=2pt}]
\tikzstyle{every node}=[draw, shape=circle, fill=black, inner sep=.5pt]

 \node at (0, 0)   (a) {};

    \node at (3, 2)   (b1) {};
    \node at (3, 0)  (b2)     {};
    \node at (3, -2)  (b3)     {};

\node at (6, 3)   (c1) {};
    \node at (6,2)  (c2)     {};
    \node at (6, 1)  (c3)     {};
    \node at (6, 0)   (c4) {};
    \node at (6, -1)  (c5)     {};
    \node at (6, -2)  (c6)     {};
    \node at (6,-3) (c7) {};

\node at (11, 6)   (d1) {};
    \node at (11,5)  (d2)     {};
    \node at (11, 4)  (d3)     {};
    \node at (11, 3)   (d4) {};
    \node at (11, 2)  (d5)     {};
    \node at (11, 1)  (d6)     {};
    \node at (11,0) (d7) {};
    \node at (11, -1)   (d8) {};
    \node at (11,-2)  (d9)     {};
    \node at (11, -3)  (d10)     {};
    \node at (11, -4)   (d11) {};
    \node at (11, -5)  (d12)     {};
    \node at (11, -6)  (d13)     {};

\node at (19, 9)   (e1) {};
\node at (19, 8)   (e2) {};
\node at (19, 7)   (e3) {};
\node at (19, 6)   (e4) {};
    \node at (19,5)  (e5)     {};
    \node at (19, 4)  (e6)     {};
    \node at (19, 3)   (e7) {};
    \node at (19, 2)  (e8)     {};
    \node at (19, 1)  (e9)     {};
    \node at (19,0) (e10) {};
    \node at (19, -1)   (e11) {};
    \node at (19,-2)  (e12)     {};
    \node at (19, -3)  (e13)     {};
    \node at (19, -4)   (e14) {};
    \node at (19, -5)  (e15)     {};
    \node at (19, -6)  (e16)     {};
    \node at (19, -7)   (e17) {};
\node at (19, -8)   (e18) {};
\node at (19, -9)   (e19) {};

\node at (21,0) (f15) {};
\node at (21.4,0) (f16) {};
\node at (21.8,0) (f17) {};


    \draw (a) -- (b1)   {};
 
   \draw(a) -- (b2)   {};
 
\draw (a)-- (b3) {};

\draw (b1) -- (c1) {};

\draw (b1) -- (c2) {};

\draw (b1) -- (c3) {};

\draw (b1) -- (c4) {};

\draw (b1) -- (c5) {};

\draw (b1) -- (c6) {};

\draw (b1) -- (c7) {};

\draw (b2) -- (c1) {};
\draw (b2) -- (c2) {};
\draw (b2) -- (c3) {};
\draw (b2) -- (c4) {};
\draw (b2) -- (c5) {};
\draw (b2) -- (c6) {};
\draw (b2) -- (c7) {};

\draw (b3) -- (c1) {};
\draw (b3) -- (c2) {};
\draw (b3) -- (c3) {};
\draw (b3) -- (c4) {};
\draw (b3) -- (c5) {};
\draw (b3) -- (c6) {};
\draw (b3) -- (c7) {};

\draw (c1) -- (d1) {};
\draw (c1) -- (d2) {};
\draw (c1) -- (d3) {};
\draw (c1) -- (d4) {};
\draw (c1) -- (d5) {};
\draw (c1) -- (d6) {};
\draw (c1) -- (d7) {};
\draw (c1) -- (d8) {};
\draw (c1) -- (d9) {};
\draw (c1) -- (d10) {};
\draw (c1) -- (d11) {};
\draw (c1) -- (d12) {};
\draw (c1) -- (d13) {};

\draw (c2) -- (d1) {};
\draw (c2) -- (d2) {};
\draw (c2) -- (d3) {};
\draw (c2) -- (d4) {};
\draw (c2) -- (d5) {};
\draw (c2) -- (d6) {};
\draw (c2) -- (d7) {};
\draw (c2) -- (d8) {};
\draw (c2) -- (d9) {};
\draw (c2) -- (d10) {};
\draw (c2) -- (d11) {};
\draw (c2) -- (d12) {};
\draw (c2) -- (d13) {};

\draw (c3) -- (d1) {};
\draw (c3) -- (d2) {};
\draw (c3) -- (d3) {};
\draw (c3) -- (d4) {};
\draw (c3) -- (d5) {};
\draw (c3) -- (d6) {};
\draw (c3) -- (d7) {};
\draw (c3) -- (d8) {};
\draw (c3) -- (d9) {};
\draw (c3) -- (d10) {};
\draw (c3) -- (d11) {};
\draw (c3) -- (d12) {};
\draw (c3) -- (d13) {};

\draw (c4) -- (d1) {};
\draw (c4) -- (d2) {};
\draw (c4) -- (d3) {};
\draw (c4) -- (d4) {};
\draw (c4) -- (d5) {};
\draw (c4) -- (d6) {};
\draw (c4) -- (d7) {};
\draw (c4) -- (d8) {};
\draw (c4) -- (d9) {};
\draw (c4) -- (d10) {};
\draw (c4) -- (d11) {};
\draw (c4) -- (d12) {};
\draw (c4) -- (d13) {};

\draw (c5) -- (d1) {};
\draw (c5) -- (d2) {};
\draw (c5) -- (d3) {};
\draw (c5) -- (d4) {};
\draw (c5) -- (d5) {};
\draw (c5) -- (d6) {};
\draw (c5) -- (d7) {};
\draw (c5) -- (d8) {};
\draw (c5) -- (d9) {};
\draw (c5) -- (d10) {};
\draw (c5) -- (d11) {};
\draw (c5) -- (d12) {};
\draw (c5) -- (d13) {};

\draw (c6) -- (d1) {};
\draw (c6) -- (d2) {};
\draw (c6) -- (d3) {};
\draw (c6) -- (d4) {};
\draw (c6) -- (d5) {};
\draw (c6) -- (d6) {};
\draw (c6) -- (d7) {};
\draw (c6) -- (d8) {};
\draw (c6) -- (d9) {};
\draw (c6) -- (d10) {};
\draw (c6) -- (d11) {};
\draw (c6) -- (d12) {};
\draw (c6) -- (d13) {};

\draw (c7) -- (d1) {};
\draw (c7) -- (d2) {};
\draw (c7) -- (d3) {};
\draw (c7) -- (d4) {};
\draw (c7) -- (d5) {};
\draw (c7) -- (d6) {};
\draw (c7) -- (d7) {};
\draw (c7) -- (d8) {};
\draw (c7) -- (d9) {};
\draw (c7) -- (d10) {};
\draw (c7) -- (d11) {};
\draw (c7) -- (d12) {};
\draw (c7) -- (d13) {};

\draw (d1) -- (e1) {};
\draw (d1) -- (e2) {};
\draw (d1) -- (e3) {};
\draw (d1) -- (e4) {};
\draw (d1) -- (e5) {};
\draw (d1) -- (e6) {};
\draw (d1) -- (e7) {};
\draw (d1) -- (e8) {};
\draw (d1) -- (e9) {};
\draw (d1) -- (e10) {};
\draw (d1) -- (e11) {};
\draw (d1) -- (e12) {};
\draw (d1) -- (e13) {};
\draw (d1) -- (e14) {};
\draw (d1) -- (e15) {};
\draw (d1) -- (e16) {};
\draw (d1) -- (e17) {};
\draw (d1) -- (e18) {};
\draw (d1) -- (e19) {};

\draw (d2) -- (e1) {};
\draw (d2) -- (e2) {};
\draw (d2) -- (e3) {};
\draw (d2) -- (e4) {};
\draw (d2) -- (e5) {};
\draw (d2) -- (e6) {};
\draw (d2) -- (e7) {};
\draw (d2) -- (e8) {};
\draw (d2) -- (e9) {};
\draw (d2) -- (e10) {};
\draw (d2) -- (e11) {};
\draw (d2) -- (e12) {};
\draw (d2) -- (e13) {};
\draw (d2) -- (e14) {};
\draw (d2) -- (e15) {};
\draw (d2) -- (e16) {};
\draw (d2) -- (e17) {};
\draw (d2) -- (e18) {};
\draw (d2) -- (e19) {};

\draw (d3) -- (e1) {};
\draw (d3) -- (e2) {};
\draw (d3) -- (e3) {};
\draw (d3) -- (e4) {};
\draw (d3) -- (e5) {};
\draw (d3) -- (e6) {};
\draw (d3) -- (e7) {};
\draw (d3) -- (e8) {};
\draw (d3) -- (e9) {};
\draw (d3) -- (e10) {};
\draw (d3) -- (e11) {};
\draw (d3) -- (e12) {};
\draw (d3) -- (e13) {};
\draw (d3) -- (e14) {};
\draw (d3) -- (e15) {};
\draw (d3) -- (e16) {};
\draw (d3) -- (e17) {};
\draw (d3) -- (e18) {};
\draw (d3) -- (e19) {};

\draw (d4) -- (e1) {};
\draw (d4) -- (e2) {};
\draw (d4) -- (e3) {};
\draw (d4) -- (e4) {};
\draw (d4) -- (e5) {};
\draw (d4) -- (e6) {};
\draw (d4) -- (e7) {};
\draw (d4) -- (e8) {};
\draw (d4) -- (e9) {};
\draw (d4) -- (e10) {};
\draw (d4) -- (e11) {};
\draw (d4) -- (e12) {};
\draw (d4) -- (e13) {};
\draw (d4) -- (e14) {};
\draw (d4) -- (e15) {};
\draw (d4) -- (e16) {};
\draw (d4) -- (e17) {};
\draw (d4) -- (e18) {};
\draw (d4) -- (e19) {};

\draw (d5) -- (e1) {};
\draw (d5) -- (e2) {};
\draw (d5) -- (e3) {};
\draw (d5) -- (e4) {};
\draw (d5) -- (e5) {};
\draw (d5) -- (e6) {};
\draw (d5) -- (e7) {};
\draw (d5) -- (e8) {};
\draw (d5) -- (e9) {};
\draw (d5) -- (e10) {};
\draw (d5) -- (e11) {};
\draw (d5) -- (e12) {};
\draw (d5) -- (e13) {};
\draw (d5) -- (e14) {};
\draw (d5) -- (e15) {};
\draw (d5) -- (e16) {};
\draw (d5) -- (e17) {};
\draw (d5) -- (e18) {};
\draw (d5) -- (e19) {};

\draw (d6) -- (e1) {};
\draw (d6) -- (e2) {};
\draw (d6) -- (e3) {};
\draw (d6) -- (e4) {};
\draw (d6) -- (e5) {};
\draw (d6) -- (e6) {};
\draw (d6) -- (e7) {};
\draw (d6) -- (e8) {};
\draw (d6) -- (e9) {};
\draw (d6) -- (e10) {};
\draw (d6) -- (e11) {};
\draw (d6) -- (e12) {};
\draw (d6) -- (e13) {};
\draw (d6) -- (e14) {};
\draw (d6) -- (e15) {};
\draw (d6) -- (e16) {};
\draw (d6) -- (e17) {};
\draw (d6) -- (e18) {};
\draw (d6) -- (e19) {};

\draw (d7) -- (e1) {};
\draw (d7) -- (e2) {};
\draw (d7) -- (e3) {};
\draw (d7) -- (e4) {};
\draw (d7) -- (e5) {};
\draw (d7) -- (e6) {};
\draw (d7) -- (e7) {};
\draw (d7) -- (e8) {};
\draw (d7) -- (e9) {};
\draw (d7) -- (e10) {};
\draw (d7) -- (e11) {};
\draw (d7) -- (e12) {};
\draw (d7) -- (e13) {};
\draw (d7) -- (e14) {};
\draw (d7) -- (e15) {};
\draw (d7) -- (e16) {};
\draw (d7) -- (e17) {};
\draw (d7) -- (e18) {};
\draw (d7) -- (e19) {};

\draw (d8) -- (e1) {};
\draw (d8) -- (e2) {};
\draw (d8) -- (e3) {};
\draw (d8) -- (e4) {};
\draw (d8) -- (e5) {};
\draw (d8) -- (e6) {};
\draw (d8) -- (e7) {};
\draw (d8) -- (e8) {};
\draw (d8) -- (e9) {};
\draw (d8) -- (e10) {};
\draw (d8) -- (e11) {};
\draw (d8) -- (e12) {};
\draw (d8) -- (e13) {};
\draw (d8) -- (e14) {};
\draw (d8) -- (e15) {};
\draw (d8) -- (e16) {};
\draw (d8) -- (e17) {};
\draw (d8) -- (e18) {};
\draw (d8) -- (e19) {};

\draw (d9) -- (e1) {};
\draw (d9) -- (e2) {};
\draw (d9) -- (e3) {};
\draw (d9) -- (e4) {};
\draw (d9) -- (e5) {};
\draw (d9) -- (e6) {};
\draw (d9) -- (e7) {};
\draw (d9) -- (e8) {};
\draw (d9) -- (e9) {};
\draw (d9) -- (e10) {};
\draw (d9) -- (e11) {};
\draw (d9) -- (e12) {};
\draw (d9) -- (e13) {};
\draw (d9) -- (e14) {};
\draw (d9) -- (e15) {};
\draw (d9) -- (e16) {};
\draw (d9) -- (e17) {};
\draw (d9) -- (e18) {};
\draw (d9) -- (e19) {};

\draw (d10) -- (e1) {};
\draw (d10) -- (e2) {};
\draw (d10) -- (e3) {};
\draw (d10) -- (e4) {};
\draw (d10) -- (e5) {};
\draw (d10) -- (e6) {};
\draw (d10) -- (e7) {};
\draw (d10) -- (e8) {};
\draw (d10) -- (e9) {};
\draw (d10) -- (e10) {};
\draw (d10) -- (e11) {};
\draw (d10) -- (e12) {};
\draw (d10) -- (e13) {};
\draw (d10) -- (e14) {};
\draw (d10) -- (e15) {};
\draw (d10) -- (e16) {};
\draw (d10) -- (e17) {};
\draw (d10) -- (e18) {};
\draw (d10) -- (e19) {};

\draw (d11) -- (e1) {};
\draw (d11) -- (e2) {};
\draw (d11) -- (e3) {};
\draw (d11) -- (e4) {};
\draw (d11) -- (e5) {};
\draw (d11) -- (e6) {};
\draw (d11) -- (e7) {};
\draw (d11) -- (e8) {};
\draw (d11) -- (e9) {};
\draw (d11) -- (e10) {};
\draw (d11) -- (e11) {};
\draw (d11) -- (e12) {};
\draw (d11) -- (e13) {};
\draw (d11) -- (e14) {};
\draw (d11) -- (e15) {};
\draw (d11) -- (e16) {};
\draw (d11) -- (e17) {};
\draw (d11) -- (e18) {};
\draw (d11) -- (e19) {};

\draw (d12) -- (e1) {};
\draw (d12) -- (e2) {};
\draw (d12) -- (e3) {};
\draw (d12) -- (e4) {};
\draw (d12) -- (e5) {};
\draw (d12) -- (e6) {};
\draw (d12) -- (e7) {};
\draw (d12) -- (e8) {};
\draw (d12) -- (e9) {};
\draw (d12) -- (e10) {};
\draw (d12) -- (e11) {};
\draw (d12) -- (e12) {};
\draw (d12) -- (e13) {};
\draw (d12) -- (e14) {};
\draw (d12) -- (e15) {};
\draw (d12) -- (e16) {};
\draw (d12) -- (e17) {};
\draw (d12) -- (e18) {};
\draw (d12) -- (e19) {};

\draw (d13) -- (e1) {};
\draw (d13) -- (e2) {};
\draw (d13) -- (e3) {};
\draw (d13) -- (e4) {};
\draw (d13) -- (e5) {};
\draw (d13) -- (e6) {};
\draw (d13) -- (e7) {};
\draw (d13) -- (e8) {};
\draw (d13) -- (e9) {};
\draw (d13) -- (e10) {};
\draw (d13) -- (e11) {};
\draw (d13) -- (e12) {};
\draw (d13) -- (e13) {};
\draw (d13) -- (e14) {};
\draw (d13) -- (e15) {};
\draw (d13) -- (e16) {};
\draw (d13) -- (e17) {};
\draw (d13) -- (e18) {};
\draw (d13) -- (e19) {};

\end{tikzpicture}

\caption{First spheres of an anti-tree, \cite{KellerRose-22b}}
\end{minipage}
\end{figure}
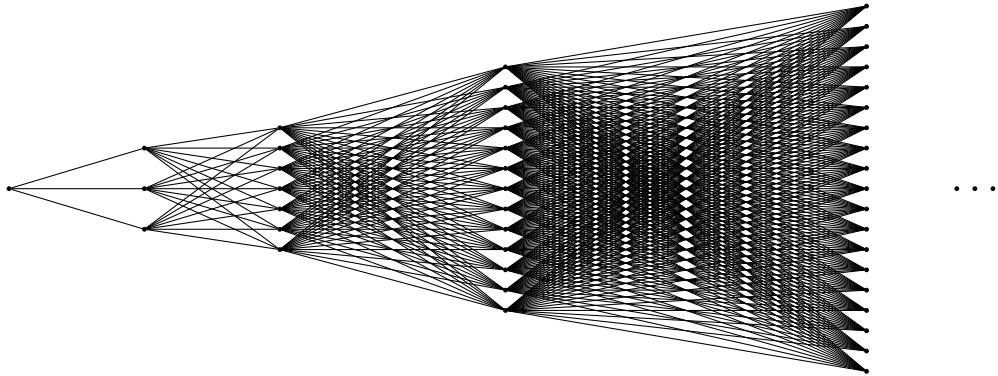
	Here, we consider graphs $ b $ with standard weights, i.e., $ b(x,y)\in\{0,1\} $, $ x,y\in X $.  If $ b $ is a connected graph over $ (X,1) $, we denote  by $$  S_{k}=\{x\in X\mid d(x,o)=k\} ,\qquad k\in\mathbb{Z}, $$ the distance spheres with respect to the combinatorial graph distance $ d $ to the root $ o\in X $. Clearly, $ S_{-k}=\emptyset $ for $ k\in \NN $. 
	
	For a so-called sphere function $ s:\NN  \longrightarrow \NN$ the corresponding anti-tree is the graph $ b $ with standard weights over $(X,10$
	such that the distance spheres satisfy
	$$  \#S_{k}=s_{k}  $$ for $ k\ge 1$,
	\begin{align*}
		b(x,y)=b(y,x)=1,\qquad x\in S_{k},y\in S_{k+1}
	\end{align*}
	for $ k\ge0  $ and $ b(x,y)=0 $ otherwise. 
	For convenience, we let $ s_{0}=1 $ and $ s_{-1}=0 $.

\begin{theorem}[{\cite[Theorem~4.9]{KellerRose-22b}}]\label{thm:antitree1} Let $ \gamma\in (0,2) $ and $X$ an anti-tree with root $o\in X$ and sphere function
\[
	s_k=[k^\gamma],\qquad k\in\NN_0.
	\] 
	Denote by $\rho$ the path degree metric with $S=1$ given in Example~\ref{ex:pathdegreemetric}.
  For $ n=\frac{4(\gamma+1)}{2-\gamma}> 2 $, there exists a constant $C>0$ such that for all $x,y\in X$ with $ |x|\neq |y| $ and $t\geq 2\cdot 72^2$ we have  
	\begin{multline*}
		p_{t}(x,y)
		\leq 
		C
		\left(1+\frac{\rho(o,x)^2+\rho(o,y)^2}{t}\right)^{\frac{n}{2}}
\\				\cdot
		\frac{\left(1\vee (\sqrt{t^2+\rho(x,y)^2}-t)\right)^{\frac{n}{2}}}{m(B_{o}(\sqrt t))}
		\exp\left(
		-t\zeta\left(\frac{\rho(x,y)}{t}\right)\right),
	\end{multline*}
and for the combinatorial distance $d$ for $t>2|d(o,x)^{(2-\gamma)/2}-d(o,y)^{(2-\gamma)/2}|^{2}$
\begin{multline*}
			p_{t}(x,y)
	\leq 
	C 
	\left(1+\frac{ d(o,x)^{2({\gamma+1})}+d(o,y)^{2(\gamma+1)}}{t^{d}}\right)
\\	\cdot \frac{ 1}{t^{\frac{d}{2}}}
	\exp\left(-\frac{ |d(o,x)^{(2-\gamma)/2}-d(o,y)^{(2-\gamma)/2}|^{2}}{Ct}\right).
\end{multline*}
\end{theorem}

A particular feature of Theorem~\ref{thm:antitree1} is the dependence of the heat kernel upper bound with respect to a root, which sometimes is referred to anchored upper heat kernel bounds. The terminology is taken from \cite{AndresDS-16,MourratOtto-16}. The technique relies on the rotational symmetry of the anti-tree and hence a comparison principle of the heat kernel and the heat kernel on a model graph over the natural numbers. Heat kernel bounds on such model graphs will investigated in more detail in \cite{KKNR-25}.
\section{Localized off-diagonal bounds via intrinsic metrics}\label{sec:localized}

In the following, the most recent results on localized off-diagonal upper heat kernel bounds obtained in \cite{KellerRose-22a, KellerRose-24, Rose-24} are presented. Originally, such characterization in terms of relative Faber-Krahn inequalities goes back to Grigor'yan, \cite{Grigoryan-94,Grigoryan-09}. The results presented in this section are a version of its discrete counterpart.

The motivation for considering localized bounds originates from the existence of manifolds with non-negative Ricci curvature with a heat kernel obeying different bounds for short and large times. In particular, they do not satisfy a uniform Nash inequality as presented in Theorem~\ref{thm:ckkw}. A standard example is given by $\RR\times \mathbb{S}^1$, where the heat kernel is just the product of the Gauss-Weierstrass function and the heat kernel of the unit circle. For large times, the heat kernel is basically the former, while for small times, the latter. 

Localized bounds depend on the position a heat kernel is evaluated at. This will be of particular importance for graphs with unbounded geometry, as it turns out that the upper heat kernel bound is only affected by the vertex degree of the arguments.

Below we obtain localized versions of Theorem~\ref{thm:ckkw} in terms of relative Faber-Krahn inequalities, which are known to be equivalent to scale-invariant Nash and Sobolev inequalities, cf.~\cite{BakryCLS-95, Barlow-book}. Distances will be measured with respect to intrinsic distances. In fact, given relative Faber-Krahn inequalities in balls, we obtain off-diagonal localized Gaussian upper heat kernel bounds with an error depending on the vertex degree at the evaluated vertices. This error will become negligible for large times. Conversely, localized upper heat kernel bounds together with volume doubling with a uniform constant imply relative Faber-Krahn inequalities in balls. There, the resulting Faber-Krahn dimension depends on the doubling dimension and the measure of the considered ball relative to the worst value the measure assigns to a single vertex inside this ball. This error again becomes neglible for large balls if it does not grow too fast.

In \cite{Rose-24}, the main theorems characterize large-scale heat kernel upper bounds in terms of large-scale Faber-Krahn inequalities. One main advantage of such results is that they apply to random models, which typically satisfy geometric assumptions only in large enough balls with high probability, \cite{Barlow-04}. Our contribution here is the validity of Gaussian upper bounds for all times instead of a restricted time interval if Faber-Krahn inequalities hold only in large balls. Conversely, an upper heat kernel bound for all times yields Faber-Krahn for all sufficiently large balls.
\\

The main result of this section will be formulated in terms of an intrinsic metric $\rho$ with jump size $S>0$.
We assume for simplicity that the intrinsic diameter satisfies $$\diam_\rho(X)=\sup\{\rho(x,y)\colon x,y\in X\}=\infty.$$ 
Note that in the case $m=\deg$, the combinatorial distance $d$ is automatically intrinsic with jump size one.
As usual, we denote distance balls with respect to $\rho$ by $$ B_x(r):=B_x^\rho(r)=\{y\in X\mid \rho(x,y)\leq r\}$$
for $r\geq 0 $, $x\in X$. If all distance balls with respect to an intrinsic metric $\rho$ are compact, i.e., they are finite, we say that $\rho$ has \emph{finite balls}. 
An intrinsic metric $ \rho  $ on a graph $b$ over $(X,m)$ is called  \emph{intrinsic path metric}  if there is $ w: X\times X\to [0,\infty] $ such that for all $ x,y\in X $
\begin{align*}
	\rho (x,y)=\inf_{x=x_{0}\sim \ldots\sim x_{n}=y}\sum_{j=1}^{n}w(x_{j-1},x_{j})
\end{align*}
{and $w(x,y)<\infty $ if and only if $ x\sim y $.}
\eat{\begin{example} The choice $ w(x,y)=(\Deg_{x}\vee\Deg_{y})^{-1/2} $ for neighbors $ x\sim y $ and $ w(x,y) =\infty$ otherwise always yields an intrinsic path metric
\end{example}}

\begin{assumption}\label{assumption2} We assume that the intrinsic metric is a path metric whose distance balls are compact and that the jump size satisfies $S<\infty$.
\end{assumption}

Assumption~\ref{assumption2} implies that 
the metric space $ (X,\rho) $ is complete and geodesic, i.e., any two vertices $ x,y\in X $ can be connected by a path $x= x_{0}\sim\ldots \sim x_{n}=y $ such that $ \rho(x,y)=\rho(x,x_{j})+\rho(x_{j},y) $ for all $ j=0,\ldots,n $, see \cite[Chapter~11.2]{KellerLW-21}.
\\

\eat{
For $f\in \cC(X)$ and $x,y\in X$ we abbreviate $\nabla_{xy}f:=f(x)-f(y)$ and set
\[
|\nabla f|^2(x):=\frac{1}{m(x)}\sum_{y\in X}b(x,y)(\nabla_{xy}f)^2.
\]
}
\eat{
Recall the first Dirichlet eigenvalue of subsets $U\subset X$ 
\[
\lambda (U)=\inf_{f\in\cC(U),f\neq 0}\frac{\sum_{x\in U}m(x) f(x)\Delta f(x)}{\sum_{x\in U} m(x)|f(x)|^2}.
\]

We will be dealing with the following properties a graph might or might not satisfy.
}


In contrast to the uniform results presented in Section~\ref{sec:uniform} and the setting of Riemannian manifolds, on graphs the implications of the characterization of heat kernel bounds involve certain error functions encoding local geometric properties. Therefore, we will introduce the error terms already in the definition of the notions involved in the characterization. 

Let $n,r\geq 0$ and functions $a,N\colon X\times [r,\infty)\to(0,\infty)$, $\Psi\colon X\times X\times [0,\infty)\to (0,\infty)$, $\Phi\colon X\times [0,\infty)\to (0,\infty)$. 
\begin{itemize}
\item[(FK)] The \emph{relative Faber-Krahn inequality} $FK(r,a,N)$ is satisfied in $X$, if for all $x\in X$, $R\geq r$, and $U\subset B_x(R)$, the first eigenvalue of the Dirichlet Laplacian $\lambda(U)$ satisfies
\[
\lambda(U):=\inf_{f\in\cC(U),f\neq 0}\frac{\sum_{x\in U}m(x) f(x)\Delta f(x)}{\sum_{x\in U} m(x)|f(x)|^2}\geq \frac{a_x(R)}{R^2}\left(\frac{m(B_x(R))}{m(U)}\right)^{\frac{2}{N_x(R)}}.
\]
\item[(G)] \emph{Gaussian upper bounds} $G(\Psi,N)$ holds in $X$ if for all $x,y\in X$ and $t\geq 0$ the heat kernel satisfies
\[
p_t(x,y)\leq \Psi_{xy}(\sqrt t)
\frac{ 
\left(1\vee \frac{t}{S^2}\arsinh^2\left(\frac{\rho(x,y) S}{t}\right)\right)^{\frac{N}{2}}
}{\sqrt{m(B_x(\sqrt t))m(B_y(\sqrt t))}}
\exp\left(-\frac{t}{S^{2}}\zeta\left(\frac{\rho(x,y)S}{t}\right)\right).
\]
\item[(VD)]The \emph{volume doubling property} $VD(\Phi,N) $ is satisfied in $X$ if for all $x\in X$ and $r\leq R$, we have 
\[
\frac{m(B_x(R))}{m(B_x(r))}\leq \Phi_x(r)\left(\frac{R}{r}\right)^{N}.
\]
\end{itemize}

Below are the error terms which will play a role in our main theorem. The conclusion of implication $\mathrm{(FK) \Rightarrow (VD) \& (G)}$ will depend on the functions $\Phi$ and $\Psi$ given first. Second, the conclusion of the implication $\mathrm{(VD) \& (G) \Rightarrow (FK)} $  will be given in terms of a new variable dimension.

Let $n>0$, $r\geq0$, $x\in X$, and define the error function $\Phi$ by 
\[
\Phi_x(s)=\Phi_x(n,r, s)=
\begin{cases}
r^n[1\vee\Deg_x]^{\frac{n}{2}+\theta(r)}&\colon s<r,\\[1.5ex]
[1\vee\Deg_x]^{\theta(s)}&\colon \text{else},
\end{cases}
\]
where we put 
\[
\theta(s)=\theta(n,s)=C_\theta\cdot s^{-\frac{1}{n+2}}, 
\]
and $C_\theta=(288S)^{\frac{1}{n+2}}(n+2)$.
For $\tau\geq0$ and $x,y\in X$, we further introduce the error function by
\[
\Psi_{xy}(\sqrt \tau)^2=\Phi_x(\sqrt{\tau_\rho})\Phi_y(\sqrt{\tau_\rho}),
\]
where we set
\[
\tau_\rho:=\tau_{\rho(x,y)}
= 
\frac{S^2}{2\arsinh^2\left(\frac{(\sqrt \tau\vee \rho(x,y))S}{\tau}\right)}.
\]
Further, we define
a dimension function $n'\colon X\times [S,\infty)\to (0,\infty)$ by
\[
n_o'(r)= n\vee \frac{\ln\left[m(B_o(Ar))\|\tfrac1m\|_{B_o(Ar)}\right]}{\ln \frac{r}{(\ln r)^{n+3}}},
\]
where $A=\exp(2^{19n}\euler)$ and $ \|f\|_{W}$ denotes the supremum norm of $f$ in $W$.

Our main theorem reads as follows.
\begin{theorem}\label{thm:main}
Let $S,n>0$, $r\geq 1000S$, and $X$ a graph as above. 
\begin{enumerate}[(i)]
\item If $X$ satisfies $FK(r,a,n)$ for some $a>0$, then there exists  $C=C({S,r,a,n})>0$ such that $X$ satisfies $G(C\Psi,n)$ and $VD(C\Phi,n)$.
\item If $X$ satisfies $G(C,n)$ and $VD(C,n)$ for some $C>0$, then there exists a positive constant $a'=a'({C,n})>0$ such that $X$ satisfies $FK(8S,a',n')$.
\end{enumerate}
\end{theorem}

\begin{remark}The proof of (ii) delivers a similar result for any $n''\ge n'$. Therefore, a good upper bound on the function $m(B_o(r))\|\tfrac1m\|_{B_o(r)}$ yields a good estimate on $n'$. A good canditate for an upper bound is given in terms of the vertex degree and the radius of the ball valid under the Faber-Krahn inequality, cf.~\cite{Rose-24}. In order to stay concise, we refrain from this discussion and refer to \cite{Rose-24} for more sophisticated results which take this observation into account.
\end{remark}
\begin{proof}$(i)$ Assume $FK(r,a,n)$.
Property $VD(C\Phi,n)$ follows from \cite[Theorem~2.6]{Rose-24}. Further, we infer from \cite[Theorem~6.3]{Rose-24} the existence of a constant $C=C(a,n,S)\geq 1$ such that for all $t\geq r^2$, we have 
\[
p_t(x,y)\leq C\ \Psi_{xy}(\sqrt t)
\frac{ 
\left(1\vee \frac{t}{S^2}\arsinh^2\left(\frac{\rho(x,y) S}{t}\right)\right)^{\frac{N}{2}}
}{\sqrt{m(B_x(\sqrt t))m(B_y(\sqrt t))}}
\exp\left(-t \ \zeta\left(\frac{\rho(x,y)S}{t}\right)\right).
\]
Next, we estimate the heat kernel for small times. To this end,
note that 
according to \cite[Lemma~2.1]{Rose-24}, there exists a constant $C=C(a,n)\geq 1$ such that 
\[
\frac{m(B_x(r))}{m(x)}\leq Cr^n\ [1\vee \Deg_x]^\frac{n}{2}
\]
for all $x\in X$. Hence, for any $t<r^2$ and $x\in X$, we have 
\[
\frac{1}{m(x)}=
\frac{m(B_x(\sqrt t))}{m(x)m(B_x(\sqrt t))}\leq \frac{m(B_x(r))}{m(x)m(B_x(\sqrt t))}\leq \frac{Cr^n \ [1\vee \Deg_x]^\frac{n}{2}}{m(B_x(\sqrt t))},
\]
and thus
\[
\frac{1}{\sqrt{m(x)m(y)}}\leq Cr^n\frac{ [1\vee \Deg_x]^\frac{n}{4}[1\vee \Deg_y]^\frac{n}{4}}{\sqrt{m(B_x(\sqrt t))m(B_y(\sqrt t))}}\leq \frac{C \ \Psi_{xy}(\sqrt \tau)}{\sqrt{m(B_x(\sqrt t))m(B_y(\sqrt t))}}.
\]
For $t\in[0,r^2]$, the above estimate applied to the universal heat kernel bound from Theorem~\ref{thm:BHY} implies
\begin{multline*}
p_t(x,y)\leq \frac{1}{\sqrt{m(x)m(y)}}\exp\left(-t \ \zeta\left(\frac{\rho(x,y)S}{t}\right)\right)
\\
\le \frac{C \ \Psi_{xy}(\sqrt \tau)}{\sqrt{m(B_x(\sqrt t))m(B_y(\sqrt t))}}\exp\left(-t \ \zeta\left(\frac{\rho(x,y)S}{t}\right)\right).
\end{multline*}
Hence, the desired property $G(C\Psi,n)$ follows by combing the above bounds for times less and larger or equal to $r$.
\\
$(ii)$ If $G(C,n)$ and $VD(C,n)$ hold in $X$, we apply \cite[Theorem~7.8]{Rose-24} with the choices $\gamma=\Phi=\Psi=1$ to obtain $FK(\exp(1\vee 32S),a',n')$ in $X$, where 
\[
a'=\left[C^{22/n}2^{38}\exp(2^{1+19n}\euler+{2}/{n})\right]^{-1}.\qedhere
\]
\end{proof}

As already mentioned in Section~\ref{sec:uniform}, since localized Faber-Krahn inequalities are equivalent to localized Nash and Sobolev inequalities, these inequalities can as well be added to the list of equivalent properties for localized Gaussian upper bounds. The drawback of the above theorem is that in order to obtain the (FK) inside a ball, one has to assume (G) and (VD) inside a much larger ball. This can be circumvented in the case of Nash and Sobolev inequalities by introducing volume corrections in their definitions, see \cite{KellerRose-22a,KellerRose-24} for purely local results in this direction.

\section{The optimal metric and Gaussian problem}\label{sec:optimal}
This section discusses the problem of optimality of the Gaussian term for the heat kernel $p$ of a given graph $b$ over the discrete measure space $(X,m)$. Since the choice of a metric is not canonic, a natural question regards the existence of an optimal metric, i.e., delivering an optimal Gaussian term. Davies' resolves this problem via a certain notion of regular graph introduced in \cite{Davies-93a}. We compare these results with another more general approach given in terms of the maximum of intrinsic metrics.
\\

Consider the integer lattice $(\ZZ^2, b, 1)$, where $b(z_1,z_2)=1$ iff $|x_1-y_1|+|x_2-y_2|=1$, $z_i=(x_i,y_i)\in\ZZ^2$, $i\in\{1,2\}$. According to \cite{Davies-93a}, by the factorization property of the heat kernel and Pang's estimates we obtain for all $z_1,z_2\in\ZZ^2$ and $t>0$
\begin{multline*}
p_t^{\ZZ^2}(z_1,z_2)=p_t^{\ZZ}(x_1,y_1)p_t^{\ZZ}(x_2,y_2)
\\
\leq \exp\left(-\frac{d(x_1,y_1)^2}{(4+\epsilon)t}-\frac{d(x_2,y_2)^2}{(4+\epsilon)t}\right)=\exp\left(-\frac{|z_1-z_2|^2}{(4+\epsilon)t}\right).
\end{multline*}

However, Davies' result presented in Section~\ref{sec:universal} transferred to $(\ZZ^2,b, 1)$ yields for all $z_1,z_2\in\ZZ^2$ and $\epsilon>0$ the existence of a $\delta>0$ such that if $d(x,y)/t<\delta$, we have
\[
p_t^{\ZZ^2}(z_1,z_2)\leq \exp\left(-\frac{d(z_1,z_2)^2}{2(4+\epsilon)t}\right).
\]
This bound is clearly worse than the former.

A way to overcome this issue was found by Davies by considering Gaussian terms involving metrics different from the combinatorial distance. For a graph $(X,b,m)$,
set 
\[
\rho_\cE(x,y)=\sup\left\{|\psi(x)-\psi(y)|\colon \psi \in \cC_c(X)\cap \cD(\cE), \|\Gamma(\psi)\|_\infty\leq 1\right\}
\]
for all $x,y\in X$.
Following the terminology of \cite{Davies-93a},
a graph $b$ is called regular if there exists a constant $S>0$ such that 
\[
|\psi(x)-\psi(y)|\leq S
\]
for all $x,y\in X$ with $b(x,y)>0$ and $\psi \in \cC_c(X)\cap \cD(\cE)$ with $\|\Gamma(\psi)\|_\infty\leq 1$. According to \cite[Lemma~5.1]{Davies-93a}, if $b$ is a graph over $(X,1)$ and $b$ is bounded below uniformly by a positive constant on the set of edges, then the graph $b$ is $S$-regular with 
$$S=\sup_{x,y\in X, b(x,y)>0}\sqrt{2m(x)/b(x,y)}.$$
For regular graphs, we have the following.
\begin{theorem}[{cf.~\cite[Theorem~5.4, Corollary~5.6]{Davies-93a}}]\label{thm:davies1}
If $S>0$ and $b$ is an $S$-regular graph over $(X,1)$, then we have
\[
p_t(x,y)\leq \frac{1}{\sqrt{m(x)m(y)}}\exp\left(-\frac{2t}{S^2}\ \zeta\left(\frac{\rho_\cE(x,y)S}{2t}\right)\right)
\]
for all $x,y\in X$ and $t>0$. In particular, for any $\epsilon>0$ there exists $\delta=\delta(S,\epsilon)>0$ such that if $\rho_\cE(x,y)/t<\delta$, then 
\[
p_t(x,y)\leq  \frac{1}{\sqrt{m(x)m(y)}}\exp\left(-
\frac{\rho_\cE(x,y)^2}{(4+\epsilon)t}\right).
\]
\end{theorem}
\begin{proof}
The bound on the heat kernel found in \cite[Theorem~5.4]{Davies-93a} reduces to
\[
p_t(x,y)\leq  \frac{1}{\sqrt{m(x)m(y)}}\exp\left(t\hat f(\rho_\cE(x,y)/t)\right),
\]
where 
\[
\hat f(\gamma)=\min_{\lambda>0} -\gamma\lambda+f(\lambda)
\]
is the Legendre transform of the function 
\[
f(\lambda)=\sup_\psi \|\Gamma(\euler^{\lambda\psi},\euler^{-\lambda\psi})\|_\infty=\sup_{\psi}\sup_{x\in X}\frac{1}{m(x)}\sum_{y\in X}b(x,y)\left(\cosh\left(\lambda(\psi(x)-\psi(y))\right)-1\right),
\]
where the first suprema run over all $\psi\in\cC_c(X)$ with $\|\Gamma(\psi)\|_\infty\leq 1$. The following argument is borrowed from \cite{BauerHuaYau-17} and applied to the present setting:
since $b$ is $S$-regular and 
\[
t\mapsto \frac{1}{t^2}(\cosh(t)-1)
\]
is increasing, we obtain
\begin{multline*}
f(\lambda)\leq \sup_\psi\sup_{x}\frac{1}{m(x)}\sum_{y}b(x,y)|\psi(x)-\psi(y)|^2S^{-2}(\cosh(\lambda S)-1)
\\
\leq 2\|\Gamma(\psi)\|_\infty S^{-2}(\cosh(\lambda S)-1)
\leq 2 S^{-2}(\cosh(\lambda S)-1).
\end{multline*}
Therefore,
\[
\hat f(\gamma)\leq \min_{\lambda>0}-\gamma\lambda +2S^{-2}(\cosh(\lambda S)-1)=-2S^{-2}\zeta(\gamma S/2).
\]
The in-particular statement follows as in \cite[Corollary~5.6]{Davies-93a}.
\end{proof}

\begin{remark}
Comparing Davies' sharp Gaussian term appearing in  the proof of Theorem~\ref{thm:davies1} with the Gaussian term in Theorem~\ref{thm:ckkw}, we have that the function $f$ is strictly less than $\Lambda(\psi)^2$. Hence, the Gaussian term in Theorem~\ref{thm:davies1} improves upon the Gaussian in Theorem~\ref{thm:ckkw}.
\end{remark}

For the combinatorial graph $(\ZZ^2, b,1 )$ introduced above, Davies shows that $\rho_\cE$ captures the rotational symmetry of $\RR^2$ where $\ZZ^2$ embeds in although this is not the case for the graph itself.
\begin{theorem}[{\cite[Theorem~6.4]{Davies-93a}}]
On $(\ZZ^2,b,1)$, we have 
\[
\rho_\cE(x,y)\sim |x-y|,\qquad |x-y|\to \infty.
\]
\end{theorem}

Davies' approach only yields a non-trivial bound for the heat kernel of a graph if the graph is regular. 
A different approach is now described in terms of the notion of intrinsic metrics.

\eat{
The proof of Theorem~\ref{thm:BHY} is very close to the proof of Theorem~\ref{thm:davies1}, where the  intrinsic property of the metric replaces the regularity assumption on the graph.
}

Since a rescaling of the combinatorial distance is intrinsic on the graph $(\ZZ^2,b,4)$ mentioned above, the same factorization problem applies. Different from  Davies' ansatz, we maximize intrinsic metrics with fixed jump size. Since there is always an intrinsic metric with given jump size, cf.~Example~\ref{ex:pathdegreemetric}, 
the estimate on the heat kernel of Theorem~\ref{thm:BHY} can be trivially improved using 
\[
\rho_S(x,y)=\sup\left\{\rho(x,y)\colon\rho \text{ is intrinsic with jump size bounded by } S\right\}
\]
for any $x,y\in X$ for given $S>0$.
\eat{ and similarly
\[
\rho_\infty(x,y)=\sup\left\{\rho(x,y)\colon\rho \text{ is intrinsic}\right\}.
\]
}
In general, $\rho_S$ is not intrinsic for any $S>0$ but a pseudo metric, and takes finite values; this follows from the connectedness of the graph together with the triangle inequality and bounded jump size of $\rho_S$. By optimizing over all intrinsic metrics with given jump size, Theorem~\ref{thm:BHY} immediately yields the following. 
\begin{corollary}\label{cor:bhy}Let $b$ be a graph over $(X,m)$ and $S>0$. Then, we have
\[
p_t(x,y)\leq \frac{1}{\sqrt{m(x)m(y)}}\exp\left(-\frac{t}{S^{2}}\zeta\left(\frac{\rho_S(x,y)S}{t}\right)\right)
\]
for all $x,y\in X$, and $t>0$. Moreover, for any $\epsilon>0$ there is $\delta=\delta(S,\epsilon)>0$ such that if $\rho_S(x,y)/t<\delta$, then 
\[
p_t(x,y)\leq \frac{1}{\sqrt{m(x)m(y)}}\exp\left(-\frac{\rho_S(x,y)^2}{(2+\epsilon)t}\right).
\]
\end{corollary}

Theorem~\ref{thm:davies1} and Corollary~\ref{cor:bhy} both provide heat kernel bounds depending on different metrics. An interesting question is whether or not these metrics coincide. A partial  answer is given by the following lemma. 
\begin{lemma}Let $S>0$ and $b$ be a graph over $(X,m)$. 
\begin{enumerate}[(i)]
\item Define 
\[
\rho_S'(x,y)=\sup\{\rho(x,y)\colon \rho \ \text{is intrinsic with finite balls and jump size bounded by } S\},
\]
for all $x,y\in X$. Then, we have 
\[
\sqrt 2\rho_S'\leq \rho_\cE.
\]
\item If $\sup\Deg<\infty$, then
\[
\sqrt 2\rho_S\leq \rho_\cE.
\]
\item If $b$ is $\sqrt 2S$-regular, then $$\sqrt 2\rho_S\geq \rho_\cE.$$ 
\end{enumerate}
\end{lemma}

\begin{proof}
\begin{enumerate}[(i)]
\item 
Fix $x,y\in X$ and let $\rho$ be intrinsic with finite balls, and jump size bounded by $S$. For $R>0$ define  $\phi_R=(1-\rho(x,\cdot)/R)_+$. Since $\rho$ is intrinsic, we have $\|\Gamma(\phi_R)\|_\infty\leq 1/2R^2$. Hence, the function $\psi_R=(R-\rho(x,\cdot))_+=R\phi_R$ satisfies
\[
\|\Gamma(\psi_R)\|_\infty=\frac{R^2}{2}\sup_{z\in X}\sum_{y\in X}b(z,y)|\phi_R(z)-\phi_R(y)|^2=R^2\|\Gamma(\phi_R)\|_\infty
\leq 
\frac{1}{2}.
\]
Thus, $\|\Gamma(\sqrt 2\psi_R)\|_\infty\leq 1$.
Since $\rho$ satisfies the finite balls condition, $\sqrt 2\psi_R\in\cC_c(X)$ for any $R>0$. Moreover, since the graph is connected and the jump size of $\rho$ is finite, we have $\rho(x,y)<\infty$. If we choose $R\geq \rho(x,y)$, we obtain
\[
\sqrt 2\rho(x,y)=\sqrt 2R-\sqrt 2(R-\rho(x,y))_+=
\sqrt 2\psi_R(x)-\sqrt 2\psi_R(y)\leq \rho_\cE(x,y).
\]
Taking the supremum with respect to $\rho$ yields the claim.
\item 
The condition on the vertex degree implies
\[
\rho_\cE(x,y)=\sup\{|\psi(x)-\psi(y)|\colon \psi\in\cC(X),\|\Gamma(\psi)\|_\infty\leq 1\}
\]
for all $x,y\in X$, \cite[Lemma~4.5]{Davies-93a}. Fix $x,y\in X$. The function $\psi_\rho=\rho(x,\cdot)$ satisfies
\[
\|\Gamma(\psi_\rho)\|_\infty=\frac{1}{2}\sup_{z\in X}\sum_{y\in X}b(z,y)|\rho(x,z)-\rho(x,y)|^2
\leq 
\frac{1}{2}\sup_{z\in X}\sum_{y\in X}b(z,y)|\rho(x,z)|^2\leq \frac12.
\]
Therefore, $\|\Gamma(\sqrt 2\psi_\rho)\|_\infty\leq 1$ and hence
\[
\sqrt 2\rho(x,y)=|\sqrt 2\psi_\rho(x)-\sqrt2\psi_\rho(y)|\leq \rho_\cE(x,y).
\]
Taking the supremum with respect to $\rho$ yields the claim.
\item 
Let $b$ be regular with constant $S$ and $\psi\in \cC_c(X)\cap \cD(\cE)$ with $\|\Gamma(\psi)\|_\infty\leq 1$. Then, the function $$\rho_\psi\colon X\times X\to[0,\infty], \quad (x,y)\mapsto |\psi(x)-\psi(y)|$$ is a pseudometric and satisfies
\[
\sum_{y\in X}b(x,y)\rho_\psi(x,y)^2
=
\sum_{y\in X}b(x,y)|\psi(x)-\psi(y)|^2\leq  2\|\Gamma(\psi)\|_\infty\leq 2,
\]
i.e., $\rho_\psi/\sqrt 2$ is intrinsic. Since $b$ is regular with constant $\sqrt 2S$, the jump size of $\rho_{\psi}/\sqrt 2$ is bounded by $S$. Therefore,
\[
\rho_\psi(x,y)=|\psi(x)-\psi(y)|\leq \sqrt 2\rho_S(x,y)
\]
for any $x,y\in X$. Taking the supremum with respect to $\psi$ over the left-hand side yields the claim.
\qedhere
\end{enumerate}
\end{proof}

A natural question is to which extent the results presented in this note can be sharpened by the use of $\rho_S$. 
Theorem~\ref{thm:ChenFolz} clearly generalizes immediately, and so does Theorem~\ref{thm:Bella}. However, how to include $\rho_S$ in Theorem~\ref{thm:main} remains elusive.

\begin{acknowledgements}
Support by the DFG grant  "\emph{Heat kernel behavior at infinity on graphs}", project number 540199605,
is gratefully acknowledged. The author thanks Alexander Grigor'yan and Matthias Keller for drawing his attention to the problem of best metrics for heat kernel bounds and Noema Nicolussi for a remark which improved the manuscript. Moreover, he thanks the University of Bielefeld for the hospitality during his recent visit as well as the organizers of the conference "MeRiOT 2024" in Varenna for creating a stimulating conference atmosphere.
\end{acknowledgements}

\bibliographystyle{alpha}
\newcommand{\etalchar}[1]{$^{#1}$}
\end{document}